\documentclass{amsart}
\usepackage{mathrsfs}
\usepackage{amsmath}
\usepackage{epsfig}
\usepackage{amsfonts}
\usepackage{amssymb}
\usepackage{amsthm}
\usepackage{amscd}
\usepackage[all]{xy}
\usepackage{rotating}
\usepackage{lscape}
\usepackage{amsbsy}
\usepackage{verbatim}
\usepackage{moreverb}

\title{Homological obstructions to string orientations}
\author{Christopher L.~Douglas}
\thanks{The first author was supported in part by a Miller Research Fellowship.}
\address{Department of Mathematics, University of California, Berkeley, CA 94720, USA}
\email{cdouglas@math.berkeley.edu}
\author{Andr\'e G. Henriques}
\address{Mathematisch Instituut, Universiteit Utrecht, 3508 TA Utrecht, NL}
\email{a.g.henriques@uu.nl}
\author{Michael A.~Hill}
\address{Department of Mathematics, University of Virginia, Charlottesville, VA 22904, USA}
\email{mikehill@virginia.edu}

\usepackage{amsmath}
\usepackage{amsfonts}
\usepackage{amssymb}
\usepackage{amsthm}
\usepackage{comment}
\usepackage{epsfig}
\usepackage{psfrag}
\usepackage{mathrsfs}
\usepackage{amscd}
\usepackage[all]{xy}
\usepackage{amsbsy}
\usepackage{multicol}
\usepackage{latexsym}

\newtheorem{thm}{Theorem}[section]
\newtheorem{prop}[thm]{Proposition}

\newtheorem{cor}[thm]{Corollary}

\theoremstyle{definition}
\newtheorem{defn}[thm]{Definition}

\theoremstyle{remark}
\newtheorem{remark}[thm]{Remark}

\newcommand{\F}{\mathbb F}
\newcommand{\A}{\mathcal A}

\newcommand{\ing}{\includegraphics[scale=1]}
\newcommand{\ingm}{\includegraphics[scale=1]}
\newcommand{\ingn}{\includegraphics[scale=.6]}
\newcommand{\ings}{\includegraphics[scale=.65]}
\newcommand{\ingsf}{\includegraphics[height=612.2pt]}
\newcommand{\nid}{\noindent}
\newcommand{\sm}{\wedge}
\newcommand{\ra}{\rightarrow}
\newcommand{\la}{\leftarrow}
\newcommand{\nn}{\nonumber}

\newcommand{\ZZ}{\mathbb Z}

\newcommand{\RP}{\mathbb{R}\mathrm{P}}
\newcommand{\CP}{\mathbb{C}\mathrm{P}}
\newcommand{\HP}{\mathbb{H}\mathrm{P}}
\newcommand{\mP}{\mathcal P}

\begin{document}

\begin{abstract}
We observe that the Poincar\'e duality isomorphism for a string manifold is an isomorphism of modules over the subalgebra $\A(2)$ of the modulo 2 Steenrod algebra.  In particular, the pattern of the operations $Sq^1$, $Sq^2$, and $Sq^4$ on the cohomology of a string manifold has a symmetry around the middle dimension.  We characterize this kind of cohomology operation duality in term of the annihilator of the Thom class of the negative tangent bundle, and in terms of the vanishing of top-degree cohomology operations.  We also indicate how the existence of such an operation-preserving duality implies the integrality of certain polynomials in the Pontryagin classes of the manifold.
\end{abstract}

\maketitle

\section{Introduction}

The special orthogonal group $SO(n)$ is a $0$-connected cover of the orthogonal group $O(n)$, and the spin group $Spin(n)$ is in turn a $2$-connected cover of the orthogonal group.  The next stage of this progression is the string group $String(n)$, which is a $6$-connected cover of the orthogonal group.  Altogether there is a sequence of topological groups, and a corresponding sequence of classifying spaces:
\begin{gather}
O(n) \la SO(n) \la Spin(n) \la String(n) \nn \\
BO(n) \la BSO(n) \la BSpin(n) \la BString(n) \nn
\end{gather}
The latter sequence realizes the early stages of the Whitehead tower of $BO(n)$---indeed for many years, before Tom Goodwillie's elegant denomination ``$BString$", the space $BString(n)$ was simply referred to as ``$BO(n)\langle 8 \rangle$".  The string group is not a Lie group, and therefore requires some care to define; nevertheless there are by now many constructions of $String(n)$, including those based on bundle gerbes~\cite{brylinski}, 2-groups~\cite{bcss, hlinf, pries}, von Neumann algebras~\cite{stolzteichner}, and conformal field theory~\cite{dh-gss}.

A manifold $M$ admits an orientation, a spin orientation, or a string orientation if the classifying map $\tau: M \ra BO(n)$ of its tangent bundle has a lift to, respectively, the classifying space $BSO(n)$, $BSpin(n)$, or $BString(n)$.  There are corresponding notions for the stable tangent bundle: a manifold is stably orientable, stably spin orientable, or stably string orientable if for some $m$ the sum $\tau \oplus \epsilon^m$ of the tangent bundle with a trivial bundle admits a lift to $BSO(n+m)$, $BSpin(n+m)$, or $BString(n+m)$, respectively.  Note, though, that a manifold is stably orientable, stably spin orientable, or stably string orientable if and only if it is respectively orientable, spin orientable, or string orientable.  A manifold $M$ equipped with a particular lift of the map $\tau: M \ra BO(n)$ to $BString(n)$ is called a ``string manifold" or a ``string oriented manifold"; however, we sometimes abuse terminology and refer to the mere existence of such a lift by saying that ``$M$ is string".

As their name would suggest, string manifolds play a fundamental role in string theory, analogous to the role spin manifolds play in quantum mechanics.  Specifically, the propagation of a quantum particle has a global world-line anomaly unless spacetime is a spin manifold.  Similarly the propagation of a quantum string has a global world-sheet anomaly unless spacetime is string.  This observation goes all the way back to Killingback~\cite{killingback}.  Around the same time, Witten~\cite{witten-egqft} related string manifolds to elliptic cohomology by constructing a modular-form-valued invariant for string manifolds---this invariant quickly became known as ``the Witten genus".  The modern formulation~\cite{ahs, ahs-ii} of this relationship is that there is a map of commutative ring spectra $\sigma: MString \ra TMF$ from string bordism to the spectrum of topological modular forms.  In particular, an $n$-dimensional string manifold represents a class in the bordism group $MString_n$ and therefore determines a class in the coefficient group $TMF_n$---that class is the pushforward to a point of the $TMF$-fundamental class of the manifold.

As oriented manifolds exhibit Poincar\'e duality for integral homology, and spin manifolds exhibit Poincar\'e duality for real $K$-theory, so string manifolds exhibit Poincar\'e duality for $TMF$-cohomology.  In this sense, the existence of a string orientation on a manifold tends to be patently visible from its $TMF$-cohomology.  Computing $TMF$-cohomology is not, however, an easy matter.  The purpose of this note is to observe that string orientations also control duality properties in \emph{ordinary} cohomology, not through duality of cohomology groups but through duality of certain cohomology operations.  Specifically we show that in the cohomology $H^\ast(M;\F_2)$ of a string manifold $M$, the pattern of the operations $Sq^1$, $Sq^2$, and $Sq^4$ has a symmetry around the middle dimension; similarly in the cohomology $H^\ast(M;\F_3)$ the operations $\beta$ and $\mP^1$ are symmetrically distributed.  In particular, any asymmetry among these operations obstructs the existence of a string orientation.  In a great many examples, this recognition principle obviates the need to do even the simplest characteristic class computations to determine that a manifold cannot be string.

In section~\ref{sec-psd} we introduce the notion of Poincar\'e self-duality with respect to a collection of cohomology operations, and show that on a manifold this duality can be detected by examining which operations annihilate the Thom class of the negative tangent bundle.  By relating that annihilation condition to the vanishing of characteristic classes, we then present our basic recognition principle for string orientations on manifolds.  We proceed to completely characterize Poincar\'e self-duality in terms of operations into the top-degree cohomology group, and to describe integrality results for the Pontryagin classes of Poincar\'e self-dual manifolds.  In section~\ref{sec-examples} we illustrate the preceding results in a variety of examples.  In particular we include pictures of the cohomology of real, complex, and quaternionic projective spaces, and of $G_2$- and $F_4$-homogeneous spaces, which exhibit different aspects of the duality recognition principles.  The appendix contains a portrait of the subalgebra of the Steenrod algebra relevant for the cohomological duality of string manifolds.

\section{Poincar\'e self-duality and string obstructions} \label{sec-psd}

Throughout $p$ is prime, and $\A$ denotes the modulo $p$ Steenrod algebra.  Recall that any finite spectrum $X$ has a Spanier-Whitehead dual, that is a spectrum $DX$ and a map $DX \sm X \ra S$ such that the resulting homomorphisms $H_i(X; \ZZ) \ra H^{-i}(DX; \ZZ)$ and $H_i(DX; \ZZ) \ra H^{-i}(X; \ZZ)$ are both isomorphisms.

\begin{defn}
Let $\Theta = \{\theta_i, i \in I\} \subset \mspace{1mu} \A$ be a collection of Steenrod operations, and let $\A_{\Theta}$ denote the subalgebra of the Steenrod algebra generated by $\Theta$.  A finite $CW$ complex $X$ is \emph{Poincar\'e self-dual} with respect to the operations $\Theta$, if for some $n$ there exists an isomorphism
\[
H^\ast(DX; \F_p) \cong H^{\ast+n}(X; \F_p)
\]
between the cohomology of $DX$ and $X$ as modules over the subalgebra $\A_{\Theta}$.
\end{defn}
\nid We also refer to the cohomology $H^\ast(X)$ directly as being ``Poincar\'e self-dual", in case the space $X$ is Poincar\'e self-dual.

We will be particularly concerned with the following subalgebras of the Steenrod algebra:
\[
\A(k)=\begin{cases}
\langle Sq^{2^i} | \, 0\leq i\leq k\rangle & p=2, \\
\langle \beta, \mP^{p^i} | \, 0\leq i <k \rangle & p\text{ odd.}
\end{cases}
\]
We sometimes use the notation $\A(k)_2$ or $\A(k)_p$ in order to explicitly specify the prime in question.
Our first observation provides a simple criterion for Poincar\'e self-duality when the space in question is a manifold.

\begin{thm} \label{thm:ZeroThom}
If $M$ is an $n$-dimensional manifold, and $Sq^{2^i}$ on the Thom class $u\in H^{-n} (M^{-\tau};\F_2)$ is zero for all $i\leq k$, then $M$ is Poincar\'e self-dual with respect to $\A(k)$.
\end{thm}
\begin{proof}
Atiyah~\cite{atiyah} identified the Spanier-Whitehead dual $DM$ with the Thom spectrum $M^{-\tau}$ of the negative tangent bundle.  The Thom isomorphism provides a natural cohomology isomorphism $- \cup u: H^{\ast+n}(M) \cong H^\ast(M^{-\tau})$.  It therefore suffices to check that the Thom isomorphism is a module map for the algebra $\A(k)$.  The hypothesis $Sq^{2^i} u = 0$ for $i \leq k$ implies that $Sq^j u = 0$ for $0 < j < 2^{k+1}$.  As a result, for all $i \leq k$ we have
\[
Sq^{2^i}(x\cup u)=\sum_{j=0}^{2^i} Sq^{2^i-j}(x)\cup Sq^j(u)=Sq^{2^i}(x)\cup u. \qedhere
\]
\end{proof}
\nid The analogous result holds at odd primes: if the operations $\beta, \mP^1, \mP^p, \ldots, \mP^{p^{k-1}}$ annihilate the Thom class $u \in H^{-n} (M^{-\tau};\F_p)$ of the negative tangent bundle of a manifold $M$, then $M$ is Poincar\'e self-dual with respect to $\A(k)_p$.

Note that the homology $H_\ast(X; \F_p)$ of a space has a left action of the Steenrod algebra, defined by $\langle \chi(Sq^i) x, a \rangle = \langle x, Sq^i a \rangle$.  Here $\langle -, - \rangle: H^\ast(X; \F_p) \otimes H_\ast(X; \F_p) \ra \F_p$ is the Kronecker pairing, and $\chi$ is the canonical antiautomorphism.  Indeed, the cohomology $H^\ast(DX;\F_p)$ of the Spanier-Whitehead dual is isomorphic to the homology $H_{-\ast}(X;\F_p)$ as a module over the Steenrod algebra.

\begin{cor}
If $M$ is a manifold and the algebra $\A(k)$ annihilates the Thom class $u \in H^{-n}(M^{-\tau}; \F_p)$, then the Poincar\'e duality isomorphism $H^{\ast+n}(X;\F_p) \cong H_{-\ast}(X;\F_p)$ is an isomorphism of left $\A(k)$-modules.
\end{cor}

It is a familiar fact that in an oriented manifold, the nontrivial Bockstein homomorphisms $\beta$ are distributed symmetrically around the middle dimension---see for example the picture of the modulo 3 cohomology of $F_4/G_2$ in Figure~\ref{pic-3}.  Using Theorem~\ref{thm:ZeroThom}, we derive analogous results for spin and string manifolds, as follows.

\begin{prop} \label{prop-spin}
A manifold $M$ is spin if and only if the cohomology $H^\ast(M; \F_2)$ is Poincar\'e self-dual with respect to $Sq^1$ and $Sq^2$, that is with respect to $\A(1)$.
\end{prop}

\begin{proof}
As before, let $\tau$ denote the tangent bundle of the $n$-dimensional manifold $M$.  If $M$ is spin then $w_1(\tau)$ and $w_2(\tau)$ both vanish.  This implies that the inverse Stiefel-Whitney classes $w_1(-\tau)=w_1(\tau)$ and $w_2(-\tau) = w_1(\tau)^2 + w_2(\tau)$ also vanish.  Thom's identity $w_i(-\tau) \cup u = Sq^i u$ (see~\cite{milnorstash}) ensures that Theorem~\ref{thm:ZeroThom} applies.  Conversely, if the cohomology $H^\ast(M;\F_2)$ is Poincar\'e self-dual, then in particular the operations $Sq^1$ and $Sq^2$ are zero respectively on $H^{n-1}$ and $H^{n-2}$.  This implies (compare the proof of Theorem~\ref{thm-topops} below) that the Wu classes $v_1$ and $v_2$ vanish.  The relations $w_1(\tau) = v_1$ and $w_2(\tau) = v_2 + Sq^1 v_1$ provide the result.
\end{proof}

\begin{prop} \label{prop-string}
If the manifold $M$ is string, then $H^\ast(M;\F_2)$ is Poincar\'e self-dual with respect to $\A(2)_2$, and $H^\ast(M;\F_3)$ is Poincar\'e self-dual with respect to $\A(1)_3$.
\end{prop}

\begin{proof}
By the assumption, the classes $w_1(\tau)$, $w_2(\tau)$, and $\frac{p_1}{2}(\tau)$ vanish.  The reduction modulo 2 of $\frac{p_1}{2}$ is $w_4$.  As in the spin case $w_1(-\tau)$ and $w_2(-\tau)$ vanish, and similarly $w_4(-\tau) = w_1(\tau)^4 + w_1(\tau)^2 w_2(\tau) + w_2(\tau)^2 + w_4(\tau)$ vanishes.  Again Thom's identity handles the rest.

All oriented manifolds are Poincar\'e self-dual with respect to the odd Bocksteins.  For the remainder of $\A(1)_3$, note that the modulo 3 reduction $\left[\frac{p_1}{2}\right]$ of $\frac{p_1}{2}$ satisfies $\left[\frac{p_1}{2}\right](-\tau) = - \left[\frac{p_1}{2}\right](\tau)$ and, by Wu~\cite{wu}, there is a 3-primary Thom identity $\left[\frac{p_1}{2}\right](-\tau) \cup u = - \mP^1 u$.  
\end{proof}

\begin{remark} \label{remark-converse}
The converse to Proposition~\ref{prop-string} is not true.  For example, the cohomologies of $\CP^{11}$ and $\HP^7$---see Figure~\ref{pic-1}---are both Poincar\'e self-dual with respect to $\A(2)_2$ and $\A(1)_3$, but neither manifold is string.
\end{remark}

Roughly speaking, we have seen that oriented manifolds have a $Sq^1$ duality, spin manifolds have a $Sq^2$ duality, and string manifolds have a $Sq^4$ duality.  This progression continues one stage further.  A manifold $M$ is called ``5-brane" if the classifying map $\tau: M \ra BO$ of its tangent bundle admits a lift to the connected cover $BO \langle 9 \rangle$.  Such manifolds exhibit a $Sq^8$ duality, as follows.

\begin{prop} \label{prop-5brane}
If the manifold $M$ is 5-brane, then $H^\ast(M;\F_2)$ is Poincar\'e self-dual with respect to $\A(3)_2$, and $H^\ast(M;\F_5)$ is Poincar\'e self-dual with respect to $\A(1)_5$.
\end{prop}

\begin{proof}
A string manifold $M$ is 5-brane if and only if the composite $\tau: M \ra BString \xrightarrow{\phi} K(\ZZ,8)$ is null, where $\phi$ is a generator of $H^8(BString;\ZZ)$.  The Serre spectral sequence for the fibration $K(\ZZ,3) \ra BString \ra BSpin$ shows that the nontrivial element in $H^8(BString;\F_2)$, that is the modulo 2 reduction of $\phi$, is $w_8$.  It follows that the Stiefel-Whitney classes $\{w_i(\tau) \, | \, i \leq 8\}$ and therefore the inverse Stiefel-Whitney classes $\{w_i(-\tau) \, | \, i \leq 8\}$ all vanish, for $\tau$ the tangent bundle of a 5-brane manifold $M$.  The Thom identity then ensures that $Sq^8$ annihilates the Thom class of the negative tangent bundle of $M$, and Theorem~\ref{thm:ZeroThom} provides Poincar\'e self-duality with respect to $\A(3)_2$.

The modulo 5 Serre spectral sequence for $K(\ZZ,3) \ra BString \ra BSpin$ shows that the reduction $[p_2]$ of the second Pontryagin class generates $H^8(BString;\F_5)$.  The classes $[p_1] (-\tau) = - [p_1] (\tau)$ and $[p_2] (-\tau) = - [p_2] (\tau) + [p_1] (\tau)^2$ therefore both vanish when the manifold in question is 5-brane.  The 5-primary Thom identity $\left[p_1^2 + 3 p_2\right] (-\tau) \cup u = \mP^1_5 u$ implies that $\mP^1_5 u = 0$, as needed.
\end{proof}

\begin{remark}
The pullback of the second Pontryagin class $p_2$ to $BString$ is $6$ times the 5-brane obstruction class $\phi \in H^8(BString;\ZZ)$; we therefore refer to $\phi$ as ``$\frac{p_2}{6}$".  The 6-divisibility of $p_2$ can be determined by comparing $H^8(BString;\ZZ)$ to $H^8(BU\langle 8 \rangle;\ZZ)$, via either the complexification map $BString \ra BU\langle 8 \rangle$ or the inclusion $BU\langle 8 \rangle \ra BString$; the Chern class $c_4$ is $6$-divisible in $BU\langle 8 \rangle$ by a direct Serre spectral sequence computation.  The general divisibility of the Pontryagin classes $p_k$ in $H^{4k}(BO\langle 4k \rangle;\ZZ)$, namely $(2k-1)!$ if $k$ is even and $2 \cdot (2k-1)!$ if $k$ is odd, is due to Bott~\cite{bott}.

Note that $\frac{p_2}{6}$ is a more intricate obstruction than $\frac{p_1}{2}$.  The latter reduces modulo $2$ to $w_4$ and modulo an odd prime $p$ to $\frac{p+1}{2} p_1$.  The former reduces modulo $2$ to $w_8$ and modulo $p>3$ to $\frac{p+1}{6} p_2$ or $\frac{1-p}{6} p_2$, but there is no primary characteristic class that controls the modulo $3$ reduction of $\frac{p_2}{6}$.  This complexity is illustrated in section~\ref{sec-examples} by the homogeneous space $F_4/G_2$.
\end{remark}

\begin{prop} \label{prop-framed}
If the manifold $M$ is stably framed, then for any prime $p$, the cohomology $H^\ast(M;\F_p)$ is Poincar\'e self-dual with respect to the entire Steenrod algebra.
\end{prop}
\begin{proof}
As modules over the Steenrod algebra, we have isomorphisms
\[
H^\ast(DM;\F_p) = H^\ast(M^{-\tau};\F_p) \cong H^\ast(\Sigma^{-n}M;\F_p) = H^{\ast+n}(M;\F_p). \qedhere
\]
\end{proof}

The above propositions describe obstructions to manifolds admitting spin, string, 5-brane, or stably framed structures, and can be reformulated as simple recognition principles---we do so in the string case:
\begin{quote}
\emph{If the pattern of $Sq^1$ and $Sq^2$ operations in the cohomology $H^\ast(M;\F_2)$ of a manifold $M$ is not symmetrically distributed about the middle dimension, or if the pattern of $Sq^4$ operations is not symmetric to the pattern of $\chi(Sq^4) = Sq^4 + Sq^2 Sq^2$ operations, then $M$ is not string.  Similarly if the pattern of $\beta$ and $\mP^1$ operations on $H^\ast(M;\F_3)$ is not symmetric, then $M$ is not string.}
\end{quote}
In particular, if the composite $Sq^2 Sq^2$ vanishes on the cohomology $H^\ast(M;\F_2)$ of a manifold $M$, and the full pattern of $Sq^1$, $Sq^2$, and $Sq^4$ operations is not symmetric around the middle dimension, then $M$ cannot be string.

No Steenrod operation on the cohomology of a space can begin in degree $0$.  Therefore, if an $n$-dimensional manifold $M$ has a nontrivial $Sq^1$, $Sq^2$, $Sq^4$, $\beta_3$, or $\mP^1_3$ with target the top-dimensional cohomology group $H^n(M)$, then the cohomology $H^\ast(M)$ has an operational asymmetry that prevents the existence of a string orientation on $M$.  In fact, the existence of an operation into the top-dimensional cohomology group completely detects whether or not the manifold is Poincar\'e self-dual with respect to the corresponding subalgebra of the Steenrod algebra.  The general $2$-primary result is as follows.

\begin{thm} \label{thm-topops}
Let $M$ be an $n$-dimensional manifold.  The operations $Sq^1, Sq^2, \ldots, Sq^{2^k}$ with target $H^n(M;\F_2)$ are all trivial if and only if $H^\ast(M;\F_2)$ is Poincar\'e self-dual with respect to $\A(k)$.
\end{thm}

\begin{proof}
The Wu classes $v_i$ realize the top-degree Steenrod operations in terms of the cup product, in the sense that
\[
\langle Sq^i x,[M]\rangle=\langle v_i\cup x, [M]\rangle.
\]
If the operations $Sq^1, Sq^2, \ldots, Sq^{2^k}$ into $H^n(M)$ vanish, then the classes $v_i$ are zero for $i < 2^{k+1}$.  The Stiefel-Whitney classes are in turn determined by the Wu classes via the relation
\[
w_i(\tau)=\sum_{a+b=i}Sq^a v_b.
\]
More concisely, the total Stiefel-Whitney class $w(\tau)$ is $Sq(v)$, the total Steenrod operation on the total Wu class.  The vanishing of the Wu classes implies that $w_i(\tau)=0$ for $i < 2^{k+1}$; the inverse classes $w_i(-\tau)$ therefore vanish in the same range.  Finally the Thom identity
\[
Sq^i u = w_i(-\tau) \cup u
\]
shows that the corresponding Steenrod operations on the Thom class vanish.  Theorem~\ref{thm:ZeroThom} ensures that the cohomology is Poincar\'e self-dual with respect to $\A(k)$.  The converse is immediate.
\end{proof}

\nid The corresponding odd-primary result is that the operations $\mP^1, \mP^p, \ldots, \mP^{p^{k-1}}$ into the top degree cohomology $H^n(M;\F_p)$ vanish if and only if the oriented manifold $M$ is Poincar\'e self-dual with respect to $\A(k)_p$.

The proof of Theorem~\ref{thm-topops} illustrates the fact that we can rephrase the vanishing of those top Steenrod operations in terms of the vanishing of characteristic classes.  For instance, in the proof of Proposition~\ref{prop-string} we already saw that a manifold is Poincar\'e self-dual with respect to $\A(2)_2$ if and only if $w_1$, $w_2$, and the modulo 2 reduction $\left[\frac{p_1}{2}\right]_2$ all vanish; similarly, Poincar\'e self-duality with respect to $\A(1)_3$ is equivalent to the vanishing of $\left[p_1\right]_3$.  Though the corresponding $p$-primary statements for $p > 3$ are less closely related to string manifolds, Poincar\'e self-duality in those cases nevertheless detects interesting integrality conditions on polynomials in the Pontryagin classes, as follows.  By ``the $k^{\text{th}}$ Newton polynomial" we will mean the polynomial $N_k$ expressing the $k^{\text{th}}$ power sum polynomial in terms of the elementary symmetric polynomials; that is, $N_k(\sigma_1, \ldots, \sigma_k) = \sum_{i} x_i^k$.

\begin{prop}
Let $p$ be an odd prime.  An oriented manifold $M$ is Poincar\'e self-dual with respect to the algebra $\A(1)_p$ if and only if the $(\frac{p-1}{2})^{\text{th}}$ Newton polynomial in the Pontryagin classes of $M$ is divisible by $p$.
\end{prop}

\begin{proof}
The $(\frac{p-1}{2})^{\text{th}}$ Newton polynomial in the modulo $p$ Pontryagin classes is called the first $p$-primary Milnor-Wu class $q_1$~\cite{milnor-ccsfs}.  This class satisfies a Wu identity $\langle \mP^1 x,[M]\rangle=\langle q_1(\tau) \cup x, [M]\rangle$, a Thom identity $\mP^1 u = q_1(-\tau) \cup u$, and the inverse relation $q_1(-\tau) = - q_1(\tau)$.  If the manifold $M$ is Poincar\'e self-dual with respect to $\A(1)_p$, then the operation $\mP^1$ into the top degree vanishes; the class $q_1(\tau)$ vanishes in turn.  Conversely, the orientability of the manifold $M$ ensures that $\beta u$ is zero; if the class $q_1(\tau)$ is zero, then $\mP^1$ annihilates the Thom class $u$ and the manifold is Poincar\'e self-dual, as desired.
\end{proof}

\nid For example, a manifold is Poincar\'e self-dual with respect to $\A(1)_5$, $\A(1)_7$, or $\A(1)_{11}$ if and only if respectively $p_1^2 + 3 p_2$ is 5-divisible, $p_1^3 + 4 p_1 p_2 + 3 p_3$ is 7-divisible, or $p_1^5 + 6 p_1^3 p_2 + 5 p_1 p_2^2 + 5 p_1^2 p_3 + 6 p_1 p_4 + 6 p_2 p_3 + 5 p_5$ is 11-divisible.  There are similar integrality conditions corresponding to Poincar\'e self-duality with respect to $\A(k)_p$ in general.  These are derived by expressing the Milnor-Wu classes $q_i$, for which $q_i(-\tau) \cup u = \mP^i u$, as polynomials in the modulo~$p$ Pontryagin classes---a manifold is Poincar\'e self-dual with respect to $\A(k)_p$ if and only if the classes $\{q_i \, | \, 1 \leq i \leq p^k - 1\}$ vanish, therefore if and only if the integral Pontryagin polynomials corresponding to those $q_i$ are $p$-divisible. 

\section{Examples} \label{sec-examples}

In this section we give a variety of examples illustrating the recognition principles described above.  Table~\ref{table-chart} lists the various manifolds we consider, in each case indicating whether the manifold is spin, string, or 5-brane, and noting what phenomenon obstructs the existence of the first of those structures it fails to possess.  The cohomologies of many of these manifolds, as modules over the appropriate subalgebras of the Steenrod algebra, are drawn in Figures~\ref{pic-1},~\ref{pic-2}, and~\ref{pic-3}---unless otherwise indicated the picture is of the modulo 2 cohomology.  In the appendix, we include a complete depiction of the algebra $\A(2)$ at the prime 2, which is helpful in computing cohomology of spaces as $\A(2)$-modules, therefore in detecting obstructions to string orientations.

\begin{table}[ht]
\centering
\caption{Tangential structures on homogeneous spaces.} \label{table-chart}
\renewcommand{\arraystretch}{1.2}
\begin{tabular}{|c|c|c|c|c|}
\hline
Space & Spin? & String? & 5-Brane? & Obstruction \\
\hline
$\RP^{4n+1}$ & No & No & No & $Sq^2$ asymmetry \\ \hline
$\RP^{8n+3}$ & Yes & No & No & $Sq^4$ asymmetry \\ \hline
$\RP^{16n+7}$ & Yes & Yes & No & $Sq^8$ asymmetry \\ \hline
$\CP^{2n}$ & No & No & No & $Sq^2$ asymmetry \\ \hline
$\CP^{6n-3}, \CP^{6n+1}$ & Yes & No & No & $\mP^1$ asymmetry \\ \hline
$\CP^{4n+1}$ & Yes & No & No & $Sq^4$ asymmetry \\ \hline
$\CP^{11}$ & Yes & No & No & $\frac{p_1}{2}=6$ \\ \hline 
$\HP^{2n}$ & Yes & No & No & $Sq^4$ asymmetry \\ \hline
$\HP^{3n}, \HP^{3n-1}$ & Yes & No & No & $\mP^1$ asymmetry \\ \hline
$\HP^7$ & Yes & No & No & $\frac{p_1}{2}=6$ \\ \hline 
$G_2/U(2)$ & No & No & No & $Sq^2$ asymmetry \\ \hline 
$G_2/SO(4)$ & Yes & No & No & $Sq^4$ asymmetry \\ \hline
$G_2/T$ & Yes & Yes & Yes & ------ \\ \hline
$F_4/Spin(9)$ & Yes & Yes & No & $Sq^8$ asymmetry \\ \hline
$F_4/G_2$ & Yes & Yes & ? & $\frac{p_2}{6}$ \\ \hline
\end{tabular}
\end{table}

In Figure~\ref{pic-1}, the real projective spaces vividly display the 2-primary obstructions.  The progression begins with $\RP^2$, which is not oriented due to a $Sq^1$ asymmetry.  The remaining cases $\RP^5$, $\RP^{11}$, $\RP^{23}$ are all $Sq^1$-symmetric, but in turn fail to be $Sq^2$-, $Sq^4$-, and $Sq^8$-symmetric, thereby fail to be spin, string, and 5-brane, respectively.  The complex projective spaces $\CP^2$ and $\CP^5$ and the quaternionic projective space $\HP^2$ are similar examples.  The spaces $\CP^3$ and $\HP^3$ provide our first 3-primary obstructions, with $\mP^1$ asymmetries preventing the existence of string orientations.  As described in Remark~\ref{remark-converse}, the spaces $\CP^{11}$ and $\HP^7$ demonstrate the failure of a converse recognition principle.

A more involved computation shows that the modulo 2 cohomology of the quotient of $G_2$ by either the short root $U(2)$ or the long root $U(2)$ appears as in Figure~\ref{pic-2}; neither, therefore, is spin.  The quotient $G_2/SO(4)$ provides perhaps the clearest and most elegant example of our string obstruction principle.  As a module over $\A(1)_2$, the cohomology $H^\ast(G_2/SO(4);\F_2)$ is symmetric, realizing in degrees $2$ through $6$ the ``joker" Steenrod module.  Still, there is an apparent $Sq^4$ operation asymmetry, and this homogeneous space is therefore not string.  Note that even if the $Sq^4$ operation in $H^\ast(G_2/SO(4);\F_2)$ were absent, this cohomology module would still fail to have ``string symmetry": the nontrivial $Sq^2 Sq^2 = \chi(Sq^4)$ in degree $2$ has no symmetrically corresponding $Sq^4$ operation.  There exist bases of the modulo 2 and modulo 3 cohomology of the $G_2$-flag manifold $G_2/T$ such that the Steenrod operations appear as in Figure~\ref{pic-2}.  These cohomologies are symmetric with respect to the full algebras $\A_2$ and $\A_3$.  This manifold is string, as can be determined directly by analyzing the fiber bundle $SU(3)/T \ra G_2/T \ra G_2/SU(3)$.  In fact, because it is the quotient of a Lie group by an abelian subgroup, the manifold $G_2/T$ is stably framed.

As with $G_2$, there are a variety of curious homogeneous spaces associated to $F_4$.  The chain of inclusions $G_2 \subset Spin(7) \subset Spin(9) \subset F_4$ is one source of such examples.  The quotient $F_4/Spin(9)$ is the octonionic projective plane---its cohomology is drawn in Figure~\ref{pic-3}.  It is evidently string, but is not 5-brane by Proposition~\ref{prop-5brane}.  The homogeneous space $F_4/G_2$ is more subtle: its cohomology---see Figure~\ref{pic-3}---is symmetric with respect to the full Steenrod algebra at every prime, but there remains a potential secondary obstruction to 5-brane orientability.

\vfill

\renewcommand{\tablename}{}

\begin{table}[!h]
\begin{center}
Notation for Figures~\ref{pic-1},~\ref{pic-2}, and~\ref{pic-3}: \vspace{8pt}
\renewcommand{\arraystretch}{1.40}
\begin{tabular}{ccccccc}
\ingn{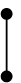} & \ingn{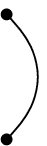} & \ingn{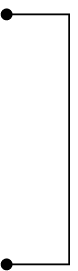} & \ingn{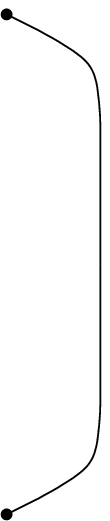} & \ingn{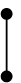} & \ingn{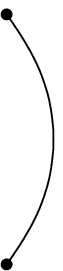} & \ingn{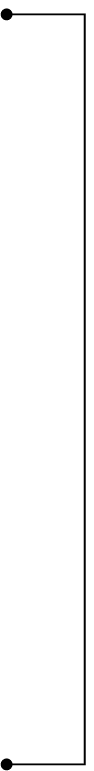} \\
\hspace{5pt} $Sq^1$ \hspace{5pt} &  \hspace{5pt} $Sq^2$ \hspace{5pt}  &  \hspace{5pt} $Sq^4$ \hspace{5pt}  &  \hspace{5pt} $Sq^8$ \hspace{5pt}  &  \hspace{5pt} $\beta$ \hspace{5pt}  &  \hspace{5pt} $\mP^1$ \hspace{5pt}  &  \hspace{5pt} $\mP^3$ \hspace{5pt} 
\end{tabular}
\end{center}
\end{table}

\renewcommand{\tablename}{Figure}

\clearpage

\topmargin=-18pt
\oddsidemargin=20pt
\evensidemargin=20pt

\thispagestyle{empty}

\begin{sidewaystable}[!p] 
\begin{center}
\renewcommand{\arraystretch}{1.25}
\caption{Cohomology of projective spaces.} \label{pic-1} \vspace{-14pt}

\begin{tabular}{|c|c|c|c|c|c|c|c|c|c|c|c|c|}

\hline
\ingm{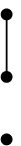} & \ingm{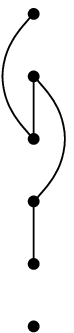} & \ingm{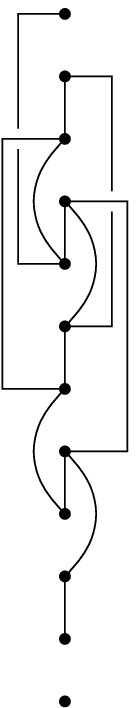} & \ingm{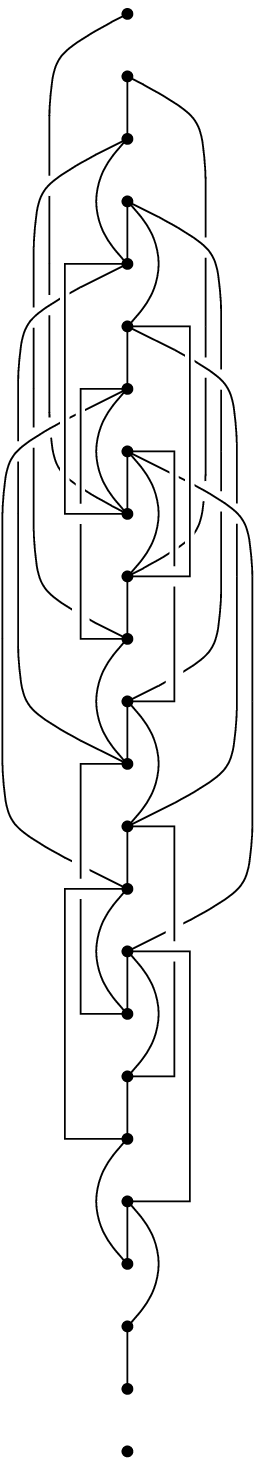} & 
\ingm{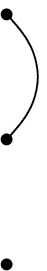} & \ingm{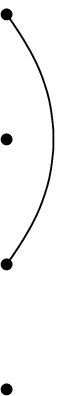} & \ingm{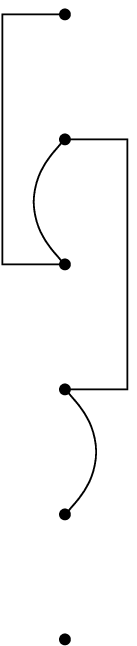} & \ingm{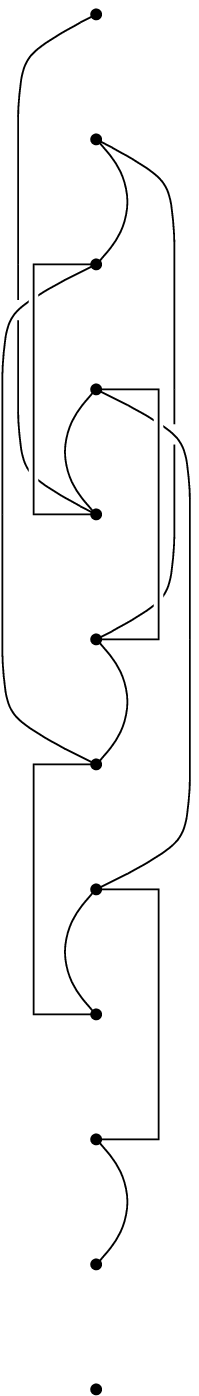} & \ingm{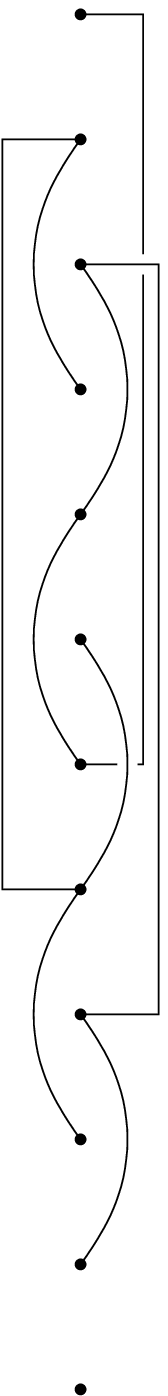} &
\ingm{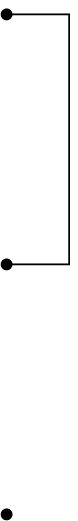} & \ingm{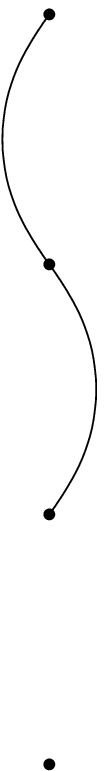} & \ingm{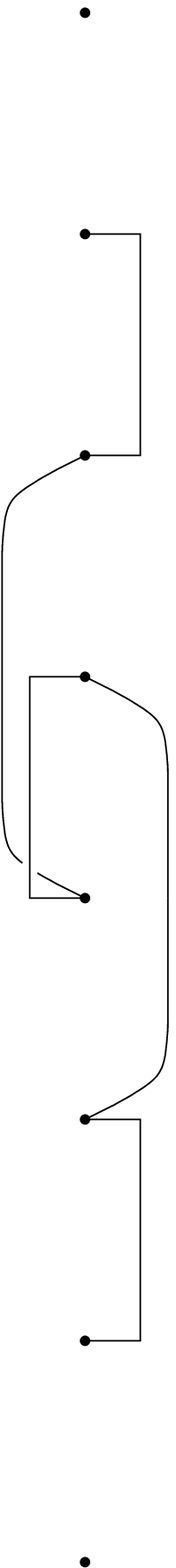} & \makebox{\rule{0pt}{517pt} \hspace{-6pt} \ingm{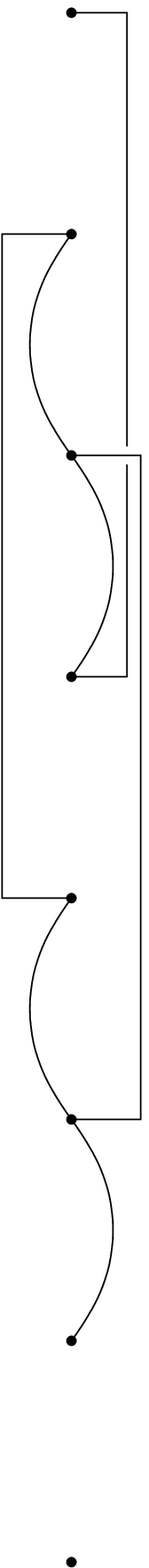}} \put(4.5,1.2){\ing{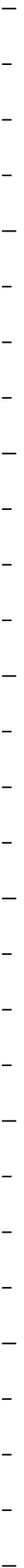}} 
\put(11,0){$\scriptstyle 0$}
\put(11,72.25){$\scriptstyle 4$}
\put(11,144.5){$\scriptstyle 8$}
\put(11,216.75){$\scriptstyle 12$}
\put(11,289){$\scriptstyle 16$}
\put(11,361.25){$\scriptstyle 20$}
\put(11,433.5){$\scriptstyle 24$}
\put(11,505.75){$\scriptstyle 28$}
\\ 
\hline
$\RP^2$ & $\RP^5$ & $\RP^{11}$ & $\RP^{23}$ & 
$\CP^2$ & $\CP^3$, $p=3$ & $\CP^5$ & $\CP^{11}$ & $\CP^{11}$, $p=3$ &
$\HP^2$ & $\HP^3$, $p=3$ & $\HP^7$ & $\HP^7$, $p=3$ \\
\hline
\end{tabular}
\end{center}
\end{sidewaystable}

\clearpage

\topmargin=1.5in
\oddsidemargin=25pt
\evensidemargin=25pt

\begin{table}[!p]
\begin{center}
\renewcommand{\arraystretch}{1.25}
\caption{Cohomology of $G_2$-homogeneous spaces.} \label{pic-2}

\begin{tabular}{|c|c|c|c|c|c|}

\hline
\ing{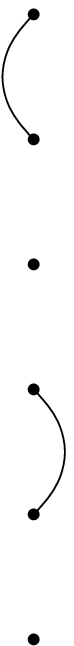} & \ing{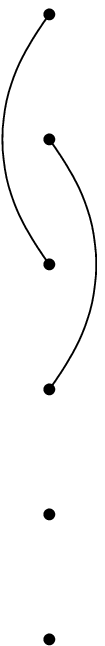} & \ing{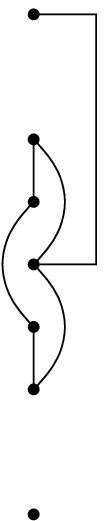} & 
\ing{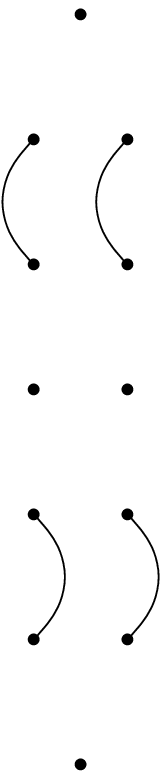} & \ing{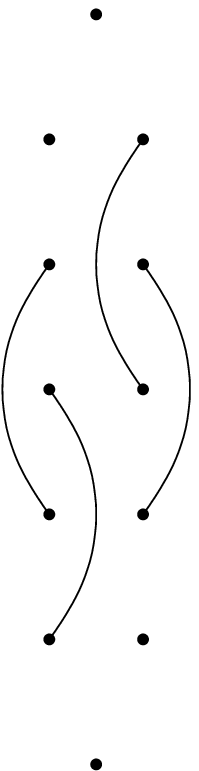} & \makebox{\rule{0pt}{263pt} \hspace{-6pt} \ing{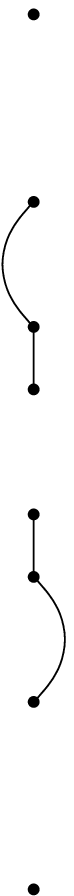}} 
\put(4.5,1.2){\ing{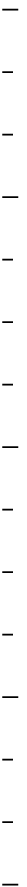}} 
\put(11,0){$\scriptstyle 0$}
\put(11,54.2){$\scriptstyle 3$}
\put(11,126.45){$\scriptstyle 7$}
\put(11,198.7){$\scriptstyle 11$}
\put(11,252.9){$\scriptstyle 14$}
\\ 
\hline
$G_2/U(2)_{s \text{ or } l}$ & $G_2/U(2)_s$, $p=3$ & $G_2/SO(4)$ & 
$G_2/T$ & $G_2/T$, $p=3$ & $G_2$ \\
\hline
\end{tabular}
\end{center}
\end{table}

\clearpage

\topmargin=-30pt

\begin{table}[!p]
\begin{center}
\renewcommand{\arraystretch}{1.25}
\thispagestyle{empty}
\caption{Cohomology of $F_4$-homogeneous spaces.} \label{pic-3}

\begin{tabular}{|c|c|c|c|c|}

\hline
\ings{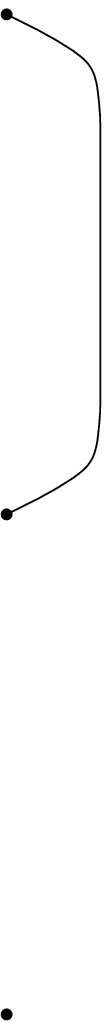} & \ings{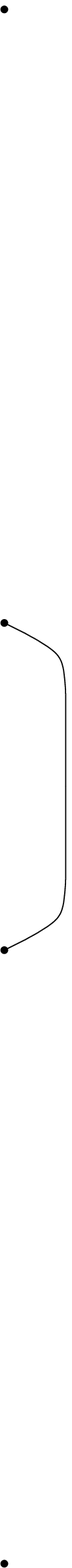} & \ings{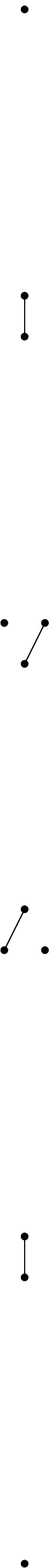} & \ings{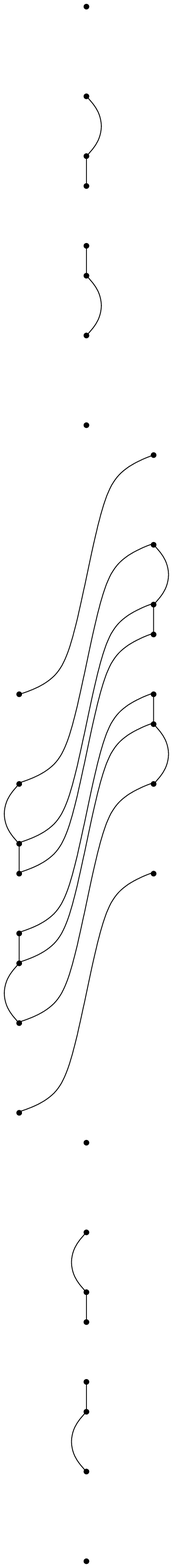} & \makebox{\rule{0pt}{621pt} \hspace{-6pt} \ings{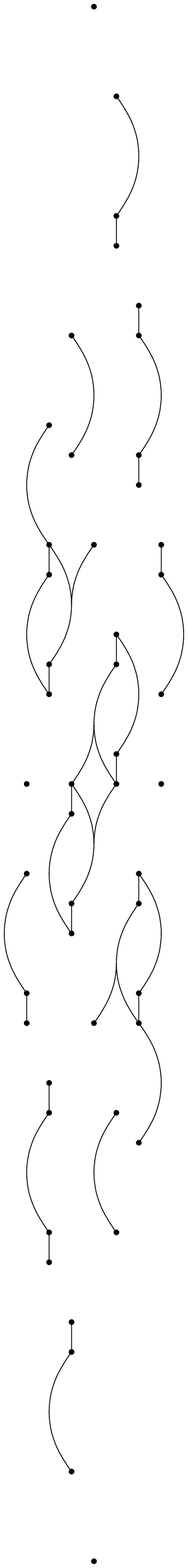}}  
\put(4.5,0.2){\ingsf{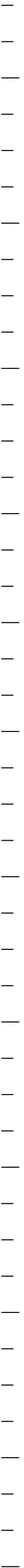}} 
\put(12,-1.5){$\scriptstyle 0$}
\put(12,33.76){$\scriptstyle 3$}
\put(12,80.72){$\scriptstyle 7$}
\put(12,127.70){$\scriptstyle 11$}
\put(12,174.70){$\scriptstyle 15$}
\put(12,221.68){$\scriptstyle 19$}
\put(12,268.66){$\scriptstyle 23$}
\put(12,303.90){$\scriptstyle 26$}
\put(12,339.14){$\scriptstyle 29$}
\put(12,386.12){$\scriptstyle 33$}
\put(12,433.10){$\scriptstyle 37$}
\put(12,480.10){$\scriptstyle 41$}
\put(12,527.08){$\scriptstyle 45$}
\put(12,574.06){$\scriptstyle 49$}
\put(12,609.3){$\scriptstyle 52$}
\\
\hline
$F_4/Spin(9)$ & \hspace{5pt} $F_4/G_2$ \hspace{5pt} & $F_4/G_2$, $p=3$ & $F_4$ & $F_4$, $p=3$ \\
\hline
\end{tabular}
\end{center}
\end{table}
\clearpage


\appendix

\setcounter{section}{-1}

\section*{Appendix. The subalgebra $\A(2)$ of the Steenrod algebra}

\topmargin=0pt
\textheight=603pt 
\textwidth=420pt  
\oddsidemargin=25pt
\evensidemargin=25pt

Here we include a portrait of the subalgebra of the 2-primary Steenrod algebra generated by $Sq^1$, $Sq^2$, and $Sq^4$.  

\vspace{.2cm}

\psfrag{1}{$1$}
\psfrag{Sq1}{$Sq^1$}
\psfrag{Sq2}{$Sq^2$}
\psfrag{Sq3}{$Sq^3$}
\psfrag{Sq4}{$Sq^4$}
\psfrag{Sq5}{$Sq^5$}
\psfrag{Sq6}{$Sq^6$}
\psfrag{Sq7}{$Sq^7$}
\begin{center}\epsfig{file=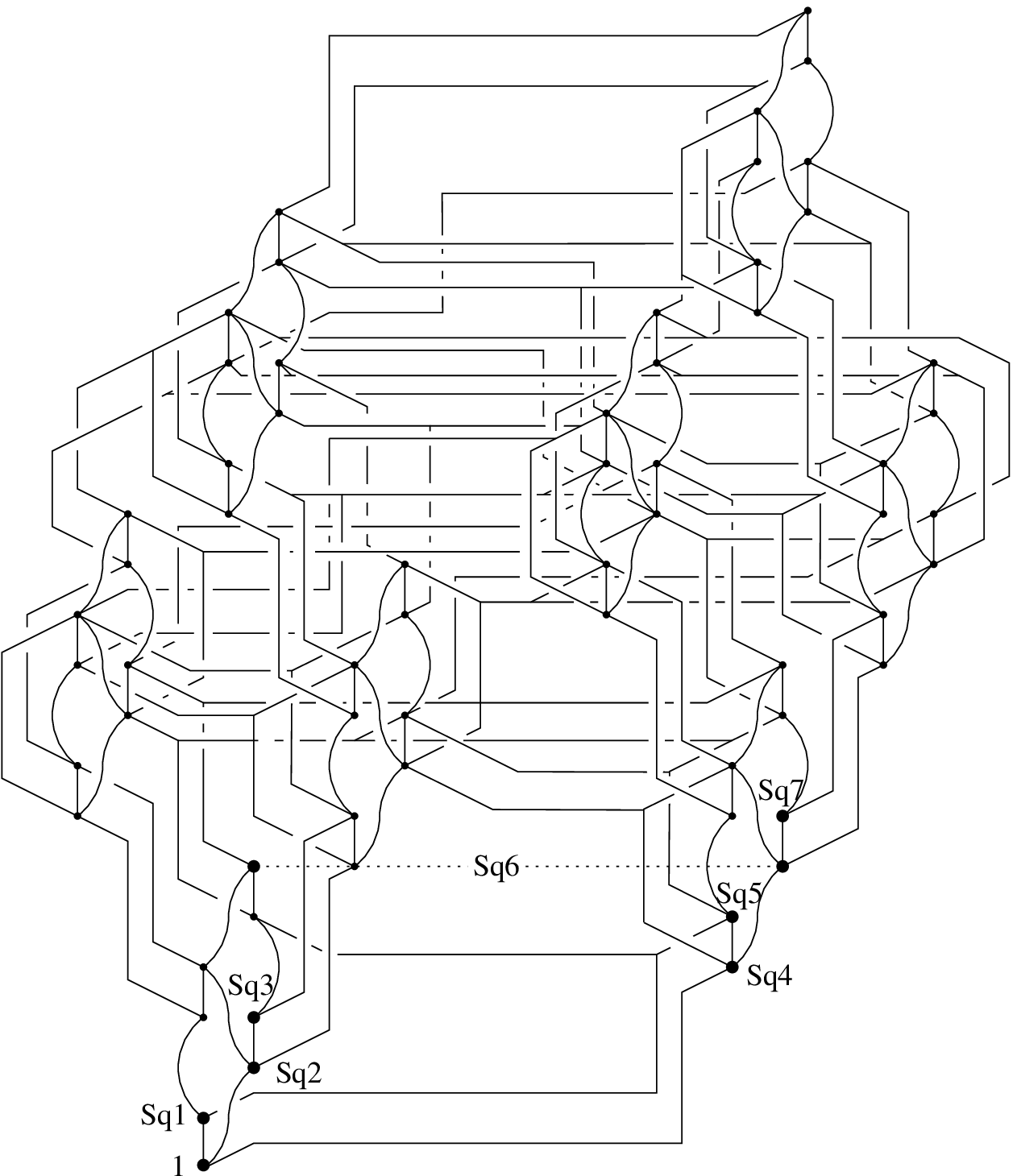, height=15cm}
\end{center}

\vspace{.2cm}

The relations are $Sq^1 Sq^1 = 0$, $Sq^2 Sq^2 + Sq^1 Sq^2 Sq^1 = 0$, $Sq^4 Sq^1 + Sq^1 Sq^4 + Sq^2 Sq^1 Sq^2 = 0$, and $Sq^4 Sq^4 + Sq^2 Sq^2 Sq^4 + Sq^2 Sq^4 Sq^2 = 0$, or pictographically, as follows:
\psfrag{+}{$+$}
\psfrag{=0}{$=0,$}
\psfrag{=4}{$=0$.}
\begin{center}\epsfig{file=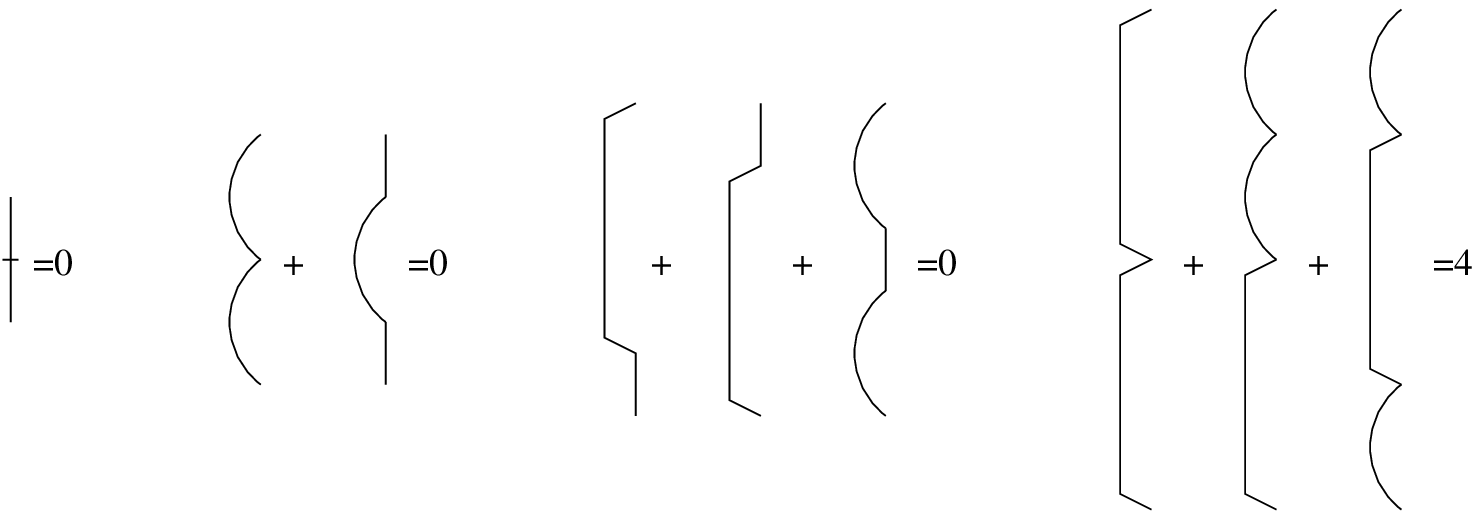, height=3cm}
\end{center}

\pagebreak

\bibliography{dhh}
\bibliographystyle{plain}

\end{document}